\documentclass[a4paper,11pt]{amsart}
\usepackage{a4wide}
\usepackage[english]{babel}
\usepackage{csquotes}
\usepackage{amsmath}
\usepackage{amsthm}
\usepackage{latexsym}
\usepackage{amssymb}
\usepackage{mathrsfs}
\usepackage{epic}
\usepackage{epsfig}
\usepackage{color}
\usepackage{dsfont}
\usepackage[colorlinks = true]{hyperref}
\usepackage{comment}
\usepackage[all]{xy}
\usepackage{enumitem}
\usepackage{eucal}
\usepackage{calc}
\usepackage{multicol}
\usepackage[all]{xy}
\usepackage[final]{showlabels}

\newcommand{\R}{\mathbb{R}}

\usepackage[style=alphabetic,sorting=nty,
    sortcites=true,doi=true,url=false,
    giveninits=true,hyperref, isbn=false]{biblatex}
\bibliography{scatt.bib}

\newcommand{\ud}{\textrm{d}}

\newcommand{\scri}{{\mathscr I}}
\newcommand{\scal}{\textrm{Scal}}

\newcommand{\som}{{\Sigma_0^{u_0<}}}

\newtheorem{theorem}{Theorem}[section]
\newtheorem*{theorem*}{Theorem}
\newtheorem{remark}[theorem]{Remark}
\newtheorem{lemma}[theorem]{Lemma}
\newtheorem{corollary}[theorem]{Corollary}

\newtheorem{definition}[theorem]{Definition}
\newtheorem{proposition}[theorem]{Proposition}

\newtheorem{assumption}{Assumption}

\subjclass[2010]{35L05, 35P25, 53A30, 35Q75}

\begin{document}

\title[Conformal Scattering]{H\"ormander's method for the characteristic Cauchy problem and conformal scattering for a non linear wave equation}
 \author{J\'er\'emie Joudioux}
 \address{Albert-Einstein-Institut, Max-Planck-Institut f\"ur Gravitationsphysik, Am M\"uhlenberg 2, 14476 Potsdam, Germany }
 \email{jeremie.joudioux@aei.mpg.de}
\urladdr{jeremiejoudioux.eu}
\maketitle

\begin{abstract} The purpose of this note is to prove the existence of a conformal scattering operator for the cubic defocusing wave equation on a non-stationary background. The proof essentially relies on solving the characteristic initial value problem by the method developed by H\"ormander. This method consists in slowing down the propagation speed of the waves to transform a characteristic initial value problem into a standard Cauchy problem. 
\end{abstract}

\section*{Introduction}

\subsection*{The result}
Consider a globally hyperbolic smooth manifold \((M,\hat g)\), whose metric $\hat g$ satisfies the vacuum Einstein equations. Let $\Sigma$ be a Cauchy hypersurface in \(M\), and $T$ be a future oriented timelike vector normal to $\Sigma$. Consider the Cauchy problem for the defocusing cubic equation
\begin{equation}\label{eq:cauchycubicwave}
\begin{aligned} 
\hat\nabla_\alpha \hat\nabla^\alpha \hat{\phi} &= \hat{\phi}^3  \\
(\hat{\phi}, T(\hat{\phi}) |_{t =0}) &\in C_0^\infty (\Sigma) \times C_0^\infty (\Sigma), 
\end{aligned}
\end{equation}
where $\hat \nabla$ is the Levi-Civita connection associated with $\hat{g}$. Since \((M,\hat g)\) is globally hyperbolic, there exists a time foliation on the manifold, and the global existence of solutions and the existence of scattering operator can be discussed for this problem. Nonetheless, since the metric is a priori not static, the standard approach to prove the global existence, for instance, through Strichartz estimates, and the existence of the scattering operator, which would require a proof of the local energy decay, might fail. The cubic wave equation \eqref{eq:cauchycubicwave} enjoys nonetheless good properties regarding conformal transformations. Indeed, if one considers a metric ${g}$ conformal to the original metric $\hat{g}$
\[
g = \Omega^2 \hat g,
 \]
one obtains the following identity
\[
\Omega^{-3}(\hat \nabla_\alpha \hat \nabla^\alpha \hat\phi - \hat\phi^3) = {\nabla}_\alpha {\nabla}^\alpha (\Omega^{-1}\hat\phi)  + \dfrac16\text{scal}_g (\Omega^{-1}\hat\phi) - (\Omega^{-1}\hat\phi) ^3, 
\]
where ${\nabla}_\alpha$ is the Levi-Civita connection associated with $g$. Hence, if $\hat\phi$ is a solution to the problem \eqref{eq:cauchycubicwave}, then $\Omega^{-1} \hat\phi$ is a solution to the corresponding geometric equation arising from the conformal metric ${g}$. This property can be exploited to prove the global existence of solutions to the problem \eqref{eq:cauchycubicwave} by considering the local existence for the conformal problem, and construct a scattering operator by relying on that operation.

Penrose introduced in the '60s a class of space-times admitting a conformal compactification. These space-times, known as asymptotically simple space-times, model isolated bodies. The compactification is obtained by completing the manifold with two disconnected hypersurfaces, null for the conformal metric when the cosmological constant vanishes, which represent the extremities of future and past null geodesics. These hypersurfaces are denoted by $\scri^+$ and $\scri^-$ respectively. The strategy for the construction of the scattering operator exploiting the existence of the conformal compactification is the following. For convenience, the construction is done for smooth data with compact support.

Consider smooth compactly supported initial data $(\hat\phi_0, \hat\phi_1)$ on $\Sigma$ for the Cauchy problem \eqref{eq:cauchycubicwave}. These data are appropriately transformed into data \(({\phi}_0, {\phi}_1)\) for the conformal cubic wave equation:
\begin{equation} \label{eq:confcauchycubicwave}
\begin{aligned}
{\nabla}_\alpha {\nabla}^\alpha (\Omega^{-1}\hat\phi)  + \dfrac16\text{scal}_g (\Omega^{-1}\hat\phi) - (\Omega^{-1}\hat\phi) ^3 =  0\\
({\phi}, \hat{T}(({\phi}))|_{t =0}  = ({\phi}_0, {\phi}_1)  \in C^\infty_0{(\hat{\Sigma})} \times C^\infty_0(\hat{\Sigma}). 
\end{aligned}
\end{equation}
Since the problem is now set on a space compact in time, the global existence problem for the Cauchy problem \eqref{eq:cauchycubicwave} turns into a much easier local existence problem \eqref{eq:confcauchycubicwave} for the rescaled equation. The existence result of Cagnac-Choquet-Bruhat \cite{MR789558} can be used to address the existence in this context. The radiation profile of the function $\hat{\phi}$ can be obtained by considering the trace of $\Omega^{-1} \hat{\phi}$. We define the mapping, generalising the inverse wave operators:
\begin{equation}
\mathfrak{T}^\pm: ({\phi}_0, {\phi}_1)  \in C^\infty_0{(\hat{\Sigma})} \times C^\infty_0(\hat{\Sigma}) \mapsto \Omega^{-1} \hat \phi |_{\scri^\pm}. 
\end{equation}
Energy estimates can be proven for the cubic wave operator and we prove in fact
 $$
  \mathfrak{T}^\pm \left( C^\infty_0{(\hat{\Sigma})} \times C^\infty_0(\hat{\Sigma})\right) \subset H^1(\scri^\pm).
  $$
The existence of the wave operators defined on \(H^1(\scri^\pm)\), inverses of the trace operators \( \left(\mathfrak{T}^\pm\right)^{-1} \), is obtained by solving the characteristic Cauchy problem for the conformal cubic wave equation with data on $\scri^{\pm}$. The generalisation of the standard scattering operator is defined as the operator
\[
\mathfrak{S}   =   \mathfrak{T}^+ \circ  \left(\mathfrak{T}^-\right) ^{-1}.
\]
In this paper, we prove the existence of a bi-Lipschitz conformal scattering operator, which extends the notion of scattering operator to space-times which are non static:
\begin{theorem*}
  There exists a bi-Lipschitz operator \(\mathfrak{S} : H^{1}(\scri^{-}) \rightarrow   H^{1}(\scri^{+}) \), which associates to past radiation profiles of solutions to the problem \eqref{eq:cauchycubicwave} the corresponding future radiation profiles.
\end{theorem*}

\subsection*{Conformal techniques, scattering, asymptotic behaviour}

The well-posedness of the Cauchy problem for this equation, and in this geometric setting, has been addressed in \cite{MR789558}. The existence of a scattering operator on Minkowski space-time has been considered, on flat space-times in \cite{Karageorgis:2006wg}, for general semi-linear wave equations, and by conformal methods on flat space-time in \cite{MR1073286}. The existence of a scattering operator, under more constraining assumptions, has been considered by the author in \cite{Joudioux:2012eo}, and this work extends this previous work to the generic cubic wave equation. In particular, we remove an artificial assumption in the decay of the coefficient of the non-linearity. The purpose of this assumption was to compensate for the blow-up of the Sobolev constant associated with the Sobolev embeddings of $H^{1}$ into $L^{6}$ in dimension $3$.

Since the metric is not stationary, this geometrical setting is a priori not amenable to standard analytic techniques to prove the existence of a scattering operator. To construct this operator, we exploit the conformal invariance of the equation, in conjunction with the conformal compactification of the space-time. The rescaled solution $\Omega^{-1}\hat\phi = \phi$ can be extended up to $\hat{M}$ by the mean of Equation \eqref{eq:confcauchycubicwave}. The trace of $\Omega^{-1}\hat\phi$ on the boundaries are the radiation profiles, in this conformal setting, discussed by Friedlander \cite{MR583989}. Hence, the trace operators $\mathfrak{T}^{\pm}$, associating to initial data for the Cauchy problem for the conformal wave equation \eqref{eq:confcauchycubicwave} to the traces of the corresponding solutions to \eqref{eq:confcauchycubicwave} on the boundaries $\scri^{\pm}$, generalises the standard inverse wave operators of the classical scattering theory, see \cite{mn04}. The actual existence of the scattering operator is performed by inverting the trace operators, that is to say, within the appropriate function spaces, solving the characteristic Cauchy problem with data on the boundaries of $M$.

Conformal methods to study global problems for partial differential equations in relativity go back to Penrose and Sachs and the peeling of higher spin fields. It was further studied in the context of the Cauchy problem in the mid-80s, early '90s (see for instance \cite{MR789558}). Friedlander \cite{MR583989} pioneered scattering theory in relativity. The reader who wishes to have a more exhaustive bibliography should refer to \cite{Joudioux:2010tn,mokdad:tel-01502657, pham:tel-01630023}. In the past years, this approach by conformal techniques to understand the asymptotic behaviour of solutions of field equations in general relativity has been studied in various contexts: Mason and Nicolas \cite{mn04} obtained the first analytic result for linear fields, followed by a peeling result on the Schwarzschild background, for scalar wave \cite{MR2461904}, spin $1/2$ and spin $1$ fields \cite{MR2888989}. This has been later extended to non-linear waves on Kerr black holes \cite{2018arXiv180108996N}. As mentioned before, the conformal scattering construction was extended to a non-linear wave equation \cite{Joudioux:2012eo}. Similar constructions based on local energy estimates \cite{2017arXiv170406441M} were obtained on Reissner-Nordstr\"om black holes \cite{2017arXiv170606993M} for the Maxwell equations. More recently, the existence of a conformal scattering operator has been proven for Yang-Mills fields on the de Sitter background \cite{2018arXiv180901559T}. 

\subsection*{Description of the work}
We now review the key points of the result. As already mentioned, this work is an extension of previous work \cite{Joudioux:2012eo}. We are in particular building on the estimates performed in the asymptotic spatial region $i^0$.

The main improvement of this paper is the possibility to handle a more general non-linearity. In the previous paper, we assumed that the non-linearity was multiplied by a decaying function $b$. This restriction has its origin in the treatment of the characteristic initial value problem on $\scri^\pm$ to construct the inverse of the trace operators $\mathfrak{T}^\pm$. The technique relies on a fixed point argument in the energy space obtained by foliating the interior of the light cones at infinity $\scri^{\pm}$. Because of the power non-linearity, Sobolev embeddings are required. Nonetheless, the volume of the leaves goes to zero, and a well-known consequence is the blow-up of the related Sobolev constant, associated with the Sobolev embedding from $H^{1}$ into $L^6$. The non-linearity is multiplied by a function decaying sufficiently fast to compensate for this blow-up of the Sobolev constant.

To circumvent that issue, the strategy to approach the characteristic initial value problem is changed. In particular, this strategy avoids the use of the Sobolev embeddings on a shrinking foliation, and relies on the work by H\"ormander \cite{MR1073287} for the wave equation. The principle is the following. One assumes the existence of a solution for the Cauchy problem for the considered geometric equation\footnote{By geometric equation, we would like to insist that the operator defining the equation is depending on the metric.}. One considers an initial characteristic surface $C$. A time function being chosen to split the metric
\[
g = -N^2 dt^2  + h_{\Sigma_t},
\]
where $h_{\Sigma_t}$ is the Riemannian metric on the leaves of the foliation induced by the time function $(\Sigma_t)$, the propagation speed of the metric is slowed down
\[
g_\lambda = - \lambda N^2 dt^2  + h_{\Sigma_t},\mbox{ for } \lambda \in \left[1/2, 1\right).
\]\
The initial characteristic hypersurface $C$ becomes spacelike for $g_{\lambda}$, and solutions to the Cauchy problem for the geometric equation associated with $g_{\lambda}$ can be solved. In the case of the cubic wave equation, this result has been obtained by Choquet-Bruhat - Cagnac \cite{MR789558}. Given initial data, a family of solutions is then obtained, depending on $\lambda$. Standard compactness arguments can then be used to prove the existence of an accumulation point as \(\lambda\) goes to one which satisfies the characteristic initial value problem. This method by H\"ormander has been extended to metrics of weak regularity in \cite{MR2244222} for the linear wave equation. This paper clarifies some points of \cite{MR2244222}, in particular in the way some energy estimates are performed and in the use of trace theorems for intermediate derivatives.

This technique by H\"ormander can only be used in a neighbourhood of timelike infinity, where the Penrose compactification leads indeed to a compact space. Near spacelike infinity $i^0$, the metric coincides with the Schwarzschild metric, and the chosen conformal factor does not lead to a compactification in the neighbourhood of $i^0$. There,  the existence of solutions to the characteristic initial value problem can be proven by relying on an explicit foliation. Global solutions to the Cauchy problem in the past of $\scri^+$ (resp. the future of $\scri^{-}$) are obtained by appropriately gluing the different solutions to the characteristic initial value problem in the neighbourhood of $i^\pm$ and the neighbourhood of $i^0$.

\subsection*{Organisation of the paper} Section \ref{sec:preliminaries} contains the precise geometric framework, in Section \ref{sec:asymptoticallysimple}, and reminders on the cubic wave equation, in Section \ref{sec:remindercubicwave}, in particular, the global existence of solutions to the defocusing wave equations, see Proposition \ref{hor3} and some estimates for the characteristic initial value problem, see Proposition \ref{prop:aprioricar}. Section \ref{sec:car} solves the characteristic initial value problem for the cubic wave equation, by H\"ormander's method, see Theorem \ref{thm:existence}. Section \ref{sec:uniqueness} addresses the proof of the uniqueness of  solutions to the Cauchy problem, see Proposition \ref{charest2}.  The next section, Section \ref{sec:neighbourhoodi0}, addresses the problem of solving the characteristic initial value problem, see Proposition \ref{prop:existenceschwarz}. The global characteristic initial value problem is solved in Section \ref{sec:globalcauchy}. Finally, the trace operators, as well as the scattering operators, are obtained in Section \ref{sec:scattering}, see Theorem \ref{thm:mainthm}, and its corollary, Corollary \ref{cor:existencescatt}. Appendix \ref{app:trace} contains reminders on a trace theorem, see in particular Theorem \ref{tracelm}, and its use when energy estimates can be proven.

\subsection*{Acknowledgements}  I want to thank Jean-Philippe Nicolas, whose strong incentive incited me to complete these notes which I drafted a long time ago.

\numberwithin{equation}{section}

\begin{sloppypar}

\section{Geometrical and analytical preliminaries}\label{sec:preliminaries}

\subsection{Conventions, notations} All along the paper, the metrics have signature $(-+++)$, and we use the Einstein summation conventions. The expression $a \lesssim b$, where $a$ and $b$ are two functions defined on $M$ means that there exists a constant $C>0$ depending on the geometry such that
$$
a \leq C b \mbox{ on }M.
$$
If $a \lesssim b$ and $b \lesssim a$, then one writes $a \approx b$.

\subsection{Regular asymptotically simple space-time}\label{sec:asymptoticallysimple}

Penrose introduced asymptotically simple space-times \cite{MR944085} as space-times, that is to say, manifolds endowed with a metric satisfying the Einstein equations, modelling a spatially localised gravitational perturbations. The existence of such space-times (solutions to the Einstein equations in vacuum with vanishing cosmological constant) can be obtained by combining the result of gluing by Corvino - Schoen and Chru\'sciel - Delay \cite{MR1920322,MR1902228,MR2225517}, and the result by Andersson - Chru\'sciel \cite{MR1269390} on the existence of space-time with smooth $\scri^{+}$ arising from regular enough data on a hyperboloid. The important point is that it generates a family of solutions to the Einstein equations which are not static, and develop nonetheless a regular $\scri$. Such solutions are by construction future and past geodesically complete \cite{MR868737}, and therefore can be seen as perturbations of Minkowski space-time, though the precise relation to any Minkowski stability result \cite{MR1946854} is yet to be clarified.

\begin{definition}[Regular asymptotically simple space-times]\label{def:asymptoticallysimple} A space-time $(\hat {M},\hat{g})$, satisfying the Einstein equations in vacuum, is a regular asymptotically simple space-time if there exists a manifold $ M$, a Lorentzian metric $g$ on $ M$ and a $C^\infty$-conformal factor $\Omega$ such that:
\begin{enumerate}
 \item $M$ is embedded in $\hat M$ and, on $M$, $ g = \Omega^2 \hat g$;
\item $ g$ and $\Omega$ are $C^\infty$ over $M$;
\item $\Omega$ is positive on $M$ and $\ud \Omega \neq 0$;
\item Any inextensible null geodesic admits a future (resp. past) endpoint on $\scri^+$ (resp. $\scri^-$).
\end{enumerate}
\end{definition}

The framework of this paper imposes the use of energy estimates to prove the existence of a solution of the Cauchy problem set up on the characteristic cones at infinity. These energy estimates usually rely on the use of Stokes' theorem, requiring that the metric is regular enough at the tip of the cone. This is why an extra regularity assumption is made at tips of the boundary of the manifold.
\begin{assumption} Let $(\hat M,\hat g)$ be regular asymptotically simple space-time whose conformal compactification is denoted by $( M,  g  = \Omega^2 \hat g)$. We assume that there exist a neighbourhood $U$ of $i^0$ in $ M$ and a system of coordinates $(t, r, \theta, \phi)$ on $U\cap \hat M$ such that, in $U\cap \hat M$, the metric is the Schwarzschild metric
$$
\tilde g = \left(1-\frac{2m}{r}\right) \ud t^2 - \left(1-\frac{2m}{r}\right)^{-1}\ud r^2 -r^2 \left(\ud\theta ^2 + \sin^2\left(\theta\right) \ud \phi^2\right).
$$
\end{assumption}

Another important assumption which is done on the structure of the asymptotically simple space-time and which is the most important restriction on the geometry:
\begin{assumption} Assume that there exists a neighbourhood \(U\) of $i^\pm$ in \({M}\), and an embedding of \(U\) into a manifold $\bar{M}$, such that the metric $g$ extends to a \(C^\infty\)-metric onto \(\bar{M}\).
\end{assumption}

\subsection{Some reminders on the cubic wave equation}\label{sec:remindercubicwave}

 \subsubsection{Conformal changes for the cubic wave equation}

 Consider the conformally related metrics \(g\) and \({\hat g}   = \Omega^{-2} g\). Then, it is a straightforward calculation that \(\hat \phi\) is a solution to the cubic wave equation on the manifold \(\hat M\)
 \[
\hat \nabla_\alpha \hat \nabla^\alpha \hat \phi +  \dfrac16 \text{scal}_{\hat g} \hat\phi = \hat \phi^3 
 \]
if, and only if, on \(\hat M\), the function $\phi = \Omega^{-1} \hat\phi$ satisfies the equation
\[
\nabla_\alpha \nabla^\alpha \phi +  \dfrac16 \text{scal}_g \phi =  \phi^3. 
\]

\subsubsection{Estimates for the defocusing wave equation} It should be noticed that we are working here with the \emph{defocusing} wave equation. Given a time foliation $(\Sigma_t)$ on $\hat M$, with unit future oriented normal $T$, the energy $E(t)$ at time $t$
\[
\Vert  \phi \Vert^4_{H^1(\Sigma_t)} + \Vert  D\phi \Vert^2_{L^2(\Sigma_t)}   + \Vert T(\phi) \Vert^2_{L^2(\Sigma_t)}  \lesssim \Vert D \phi \Vert^2_{L^2(\Sigma_t)} + \Vert T(\phi)   \Vert^2_{L^2(\Sigma_t)} + \Vert \phi \Vert^4_{L^4(\Sigma_t)} = E(t).
\]
is approximately conserved along the evolution, in the sense that 
$$
E(t) \lesssim E(0). 
$$
Hence, for the defocusing wave equation, it is possible to prove global existence for large data in $H^1(\Sigma_0)\times L^2(\Sigma_0)$.

The key tool is a Sobolev embedding from $H^1(\Sigma_t)$ into $L^4(\Sigma_t)$ in dimension 3. As long as these embedding can be performed uniformly (for instance, when all the leaves are diffeomorphic to $R^3$, or the $3-$sphere), then the global existence for all data is ensured.

The strategy to perform the estimates is the following: \begin{itemize}
    \item whenever we are working with a finite time interval, we use the standard energy estimates for the linear wave equations, similar to \cite[Chapter 1, \S 3]{sogge}, where, for the local existence, $H^s$-norms are propagated; 
    \item we otherwise use the conservation of the energy with the nonlinear term. 
\end{itemize}

\subsubsection{The Cauchy problem for the cubic wave equation}

The section contains some preliminary known results about the global existence of solutions to the Cauchy problem for the cubic wave equation
\[
 \square \phi = \phi^3 ,
\]
where \(\hat\square = \hat \nabla_\alpha \hat \nabla^\alpha\) is the wave operator associated with the metric \(\hat g\). The well-posedness result of Cagnac and Choquet-Bruhat \cite{MR789558} is recalled:
\begin{proposition}\label{hor3}
Let $X$ be a compact manifold. Consider the Lorentzian manifold $(X\times \R, g)$ so that $X\times \{t\}$ is a family of uniformly spacelike hypersurfaces. Assume that the deformation tensor of the vector $\partial_t$ is bounded as well as a sufficient number of derivatives of the metric. \newline
The Cauchy problem
\begin{equation*}\left\{
\begin{array}{l}
\square \phi = \phi^3\\
\phi|_{X} = \phi_{0}\in H^1(S)\\
T^a\nabla_a \phi|_{X} = \psi_0 \in L^2(X)
\end{array}\right.
\end{equation*}
admits a global solution in $C^1(\R, L^2(X))\cap C^0(\R, H^1(X))$. Furthermore, the a priori estimate holds, for all $t\in \R$:
$$
E_{\phi}(\{t\}\times X)\lesssim  E_{\phi}(\{0\}\times X).
$$
\end{proposition}

\begin{remark}\label{rem:bfunction} We could add a function $b$ in the equation as follows:
\[
\square \phi = b \phi^3 ,
  \]
Cagnac and Choquet-Bruhat in \cite{MR789558} addressed the local well-posedness for this equation under the following assumptions on the function $b$:  the function $b$ is bounded on $\hat M$, one time differentiable function on $M$, admits a $C^1$-extension to $ M$ and, for a given future-oriented unit timelike vector field $T^a$ on the unphysical space-time $M$, there exists a constant $c$ such that
$$
|T^a\nabla_a b| \lesssim  b.
$$
The conformal scattering operator could be constructed in that situation in a similar fashion. This extension has very little meaning in the context of that work, and its purpose, and is therefore ignored.
\end{remark}

 \subsubsection{A priori estimates for the characteristic initial value problem}

In previous work, we have proven various a priori estimates. Since these estimates are the basis of H\"ormander's method for solving the characteristic Cauchy problem, we recall in this section the various energy estimates which were previously proved. It is important to note that, in this previous work, as noted in \cite{Joudioux:2012eo}, the a priori energy estimates are proved without the decay assumptions at the tip of the vertex on the function $b$, as considered in Remark \ref{rem:bfunction}.

The first result is a local a priori energy estimate in the neighbourhood of the tip of the vertex.  The geometric and analytic settings are the following: let $p$ be a point in $\hat M$ and $U$  an open neighbourhood of $p$ in $M$, with compact closure. Assume that one can define the past light cone globally $C^-(p)$ in $U$. Let $t$ be a time function defining a foliation $\{\Sigma_t\}_{t= 0\cdots 1 }$ of the past of $C^-(p)$. The unit future oriented normal to $\Sigma_t$ is denoted by $T^a$. On $C^-(p)$, let $l$ be a generator of future oriented null geodesics ending up at $p$. Let $n$ be a null vector field transverse to $C^-(p)$ such that
$$
T = \frac{1}{2}\left(l+n\right).
$$
The family $(l,n)$ is completed into a family $(l,n,e_1, e_2)$ so that the family $(l,e_1, e_2)$ is tangent to the cone $C^-(p)$ and $(e_1, e_2)$ is orthogonal to $(l,n)$ . The derivatives with respect to the vectors $e_1, e_2$ are denoted by $\nabla_{\mathbb{S}^2}$.\\
We consider the standard Sobolev spaces on $\Sigma_t$, $H^1(\Sigma_t)$ and $L^2(\Sigma_t)$.
The set of smooth functions on the cone  $C^-(p)$ in the future of $\Sigma_0$ with compact support away from the tip $p$ is endowed with the norm:
$$
\Vert \phi \Vert^2_{H^1(C^-(p))} = \int_{C^-(p)}\left( \vert \nabla_{l} \phi\vert ^2 + \vert\nabla_{\mathbb{S}^2}\phi \vert^2 + \vert\phi \vert^2 \right)n\lrcorner  \ud \mu[\hat g]
$$
 where $n\lrcorner  \ud \mu[\hat g]$ is the contraction of the ambient four dimensional volume form $\ud \mu[\hat g]$ with the null vector $n$ transverse to the light cone $C^-(p)$. The completion for this norm of the set of smooth functions with compact support away from $p$ is denoted by $H^1(C^{-}(p))$. Using standard energy estimates, one gets the following proposition:
\begin{proposition}\label{prop:aprioricar} Let $\phi$ be a solution of the equation
$$
\square \phi + \frac {1}{6} \text{Scal}_{g} \phi = \phi^3.
$$
The following inequalities hold:
$$
\Vert \phi \Vert_{H^1(C^-(p))}^2 + \Vert \phi \Vert^4_{L^4(C^-(p))} \lesssim \left ( \Vert \phi \Vert_{H^1(\Sigma_0)}^2 +    \Vert  T^a\nabla_a \phi\Vert^2_{L^2(\Sigma_0)}  + \Vert  \phi \Vert^4_{L^4(\Sigma_0)}   \right)
$$
and
$$
\Vert \phi \Vert_{H^1(\Sigma_0)}^2 +   \Vert  T^a\nabla_a \phi\Vert^2_{L^2(\Sigma_0)}  + \Vert  \phi \Vert^4_{L^4(\Sigma_0)}  \lesssim\left(\Vert \phi \Vert_{H^1(C^-(p))}^2 + \Vert  \phi \Vert^4_{L^4(C^-(p))} \right).
$$
Furthermore, for all $t$, one has
$$
\Vert \phi \Vert_{H^1(\Sigma_t)}^2 +    \Vert  T^a\nabla_a \phi\Vert^2_{L^2(\Sigma_t)}  + \Vert  \phi \Vert^4_{L^4(\Sigma_t)}   \lesssim \left ( \Vert \phi \Vert_{H^1(\Sigma_0)}^2 +    \Vert  T^a\nabla_a \phi\Vert^2_{L^2(\Sigma_0)}  + \Vert \phi \Vert^4_{L^4(\Sigma_0)}   \right)
$$
\end{proposition}
\begin{remark} Using Sobolev embedding from $H^1(\Sigma_0)$ into $L^4(\Sigma_0)$ and from $H^1(C^-(p))$ into $L^4(C^-(p))$, one immediately gets the same inequality without the $L^4$ norms.
\end{remark}

\section{Characteristic Cauchy problem \emph{\`a la} H\"ormander}\label{sec:car}

The purpose of this section is to establish an a priori well-posedness result of the characteristic Cauchy problem for the wave equation on a curved space-time based on the result by H\"ormander \cite{MR1073287}, extended in \cite{MR2244222} and partially used in \cite{Joudioux:2012eo} to prove the existence and uniqueness of a weak solution to the characteristic Cauchy problem in $H^1(M)$. More regularity is required when establishing the Lipschitz continuity of the wave operator. The method used by H\"ormander is based on a reduction of the propagation speed of the considered wave equation so that one can resort to the standard existence result of the wave. In the following, one restricts oneself to a light cone, but the result can be extended in the same way to arbitrary weakly lightlike hypersurfaces.

The geometric setting is the following: consider a point $p$ in $M$. One denotes by $C^-(p)$ the light cone from $p$. Consider a time function defined in the interior of $C^-(p)$. The induced foliation is denoted by $\Sigma_t$. One assumes that the slice $\Sigma_0$ is in the past of $p$ and the slice containing $p$ is denoted by $\Sigma_T$. The unit future oriented normal vector to the time slices is denoted by $T^a$. One considers a 3+1 splitting of the metric in the following form:
$$
g =-  N^2 \ud t^2 + h_{\Sigma_t}
$$
where $h_{\Sigma_t}$ is a Riemannian metric on $\Sigma_t$ and the lapse $N$ is given by:
$$
N^2 = g(\nabla t, \nabla t).
$$

The functional setting is given by:
\begin{itemize}
\item on the time slice $\Sigma_t$, one defines the energy of a function $\phi$:
$$
E_\phi(\Sigma_t) = \Vert\phi\Vert^2_{H^1(\Sigma_t)} + \Vert T^a\nabla_a\phi\Vert^2_{L^2(\Sigma_t)};
$$
\item on the light cone $C^-(p)$, the following $H^1$ norm is considered:
$$
\Vert\phi\Vert^2_{H^1(C^-(p))} = \int_{C^-(p)} \left(\nabla_l \phi^2 + |\nabla_{\mathbb{S}^2}\phi|^2+\phi^2\right) T^a\lrcorner \ud \mu[g]
$$
where $l$ is a non-vanishing generator of the null directions of $C^-(p)$ and $T^a\lrcorner \ud \mu[g]$ is the contraction of the space-time volume form with the vector field $T$. The space $H^1(C^-(p))$ is defined as being the completion of the space of smooth functions whose compact support does not contain $p$.
\end{itemize}
\begin{remark} One could have defined, like H\"ormander, $H^1 (C^-(p))$ by transporting the $H^1$ structure of a timelike slice on $C^-(p)$. It happens that the two definitions coincide~\cite{Joudioux:2012eo}.
\end{remark}

Consider the characteristic Cauchy problem
\begin{equation}\label{hor1}\left\{
\begin{array}{l}
\square \phi = \phi^3\\
\phi|_{C^-(p)} = \phi_{0}\in H^1(C^-(p)).
\end{array}\right.
\end{equation}
Note that, for the purpose of the discussion of this section, the scalar curvature term is excluded of the discussion, but can be treated in a similar fashion, exploiting the fact the curvature is bounded.

The purpose of this section is to prove that there exists a strong solution to the characteristic Cauchy problem up to the hypersurface $\Sigma_0$ in the past of $p$. One then restricts oneself to the intersection of the future of $\Sigma_0$ and the past of $p$.

Using Proposition \ref{hor3}, one can then prove the following theorem:
\begin{theorem}\label{thm:existence} The characteristic Cauchy problem
\eqref{hor1} admits a global strong solution down to $\Sigma_0$ in $C^0(\R, H^1(\Sigma_\tau))\cap C^1(\R, L^2(\Sigma_\tau))$. The following a priori estimates furthermore hold:
$$
\forall \tau\in[0,T], E_{\phi}(\Sigma_\tau)\approx E_{\phi}(C^-(p)).
$$
\end{theorem}
\begin{proof} Assuming that there exists a solution in $C^0(\R, H^1(\Sigma_\tau))$, the a priori estimates are a consequence of in \cite[Proposition 4.15]{Joudioux:2012eo}.

The proof of the theorem will require at a point the use of Sobolev embeddings. This requires to work with an extension of the foliation $(\Sigma_\tau)$ to a cylinder. One then considers a smooth isometric embedding of $(J^-(p)\cap J^+(\Sigma_0),g)$ into a compact cylinder $(U, g)$. The foliation $(\Sigma_\tau)$ is extended on the cylinder $U$ in a spacelike foliation of $U$ for the extension of the metric $g$. We still denote by $(\Sigma_\tau)$ this extension; all the leaves of this foliation are now topological 3-spheres endowed with a Riemannian metric, and as a consequence, all the Sobolev spaces $H^k(\Sigma_\tau)$ considered on the leaf $\Sigma_\tau$ are equivalent since the leaf is compact.

The proof of the well-posedness relies on the method introduced by H\"ormander \cite{MR1073287} which consists in slowing down the propagation speed of the waves. A detailed proof of H\"ormander's paper is given in\cite{MR2244222} when the metric is only Lipschitz. The following proof follows step by step this paper (and especially the scheme given in the proof of theorem 4.3), modifying it when necessary.

The reduction of the propagation speed is realised by introducing a parameter $\lambda$ in $[\frac12, 1]$ and a family of metric $g_\lambda$:
$$
g_{\lambda} = -\lambda^2 N^2 \ud t^2 +h_{\Sigma_t}
$$
so that the hypersurface $C^-(p)$ becomes spacelike for the metric $g_\lambda$.

According to Proposition \ref{hor3}, for any $\lambda$ in $[\frac12, 1]$, the Cauchy problem
\begin{equation*}\left\{
\begin{array}{l}
\square_\lambda \phi = \phi^3\\
\phi|_{C^-(p)} = \phi_{0}\in H^1(C^-(p))\\
T^a\nabla_a \phi|_{C^-(p)} =0 \in L^2(C^-(p))
\end{array}\right.
\end{equation*}
admits a solution $\phi_\lambda$ down to $\Sigma_0$ in $C^1([0,T], L^2(\Sigma_\tau))\cap C^0([0,T], H^1(\Sigma_\tau))$ such that:
$$
E_{\phi_{\lambda}} \leq c_{\lambda} E_{\phi_\lambda}(C^-(p)) = \Vert\phi_0\Vert_{H^1_{\lambda}(C^-(p))}\lesssim C \Vert\phi_0\Vert_{H^1(C^-(p))}.
$$
where $c_\lambda$ is a constant which depends continuously on the scalar curvature of $g_\lambda$. Since the interval under consideration is compact and since, as a consequence of the previous remark, $c_\lambda$ depends continuously on $\lambda$, one can replace $c_\lambda$ by its supremum over $[\frac12, 1]$. Furthermore, $C^-(p)$ is compact and, as consequence, all the Sobolev spaces associated with a smooth metric are equivalent. This includes $H^1(C^-(p))$. The family $(\phi_\lambda)_{\lambda\in [\frac12, 1]}$ is then uniformly bounded in $L^\infty([0,T], H^1(\Sigma_\tau))$:
$$
E_{\phi_{\lambda}} \lesssim \Vert\phi_0\Vert_{H^1(C^-(p))}
$$

Consider now a sequence $(\lambda_n)$ converging to $1$ and one denotes by $(\phi_n)$ the associated sequence. One denotes by $U$ the volume delimited by $\Sigma_0$ and $C^-(p)$.
\begin{remark}\label{hor4}
Before studying the convergence process, let us remind that, using a priori estimates of Proposition \ref{prop:aprioricar}, the energy on the initial time slice $\Sigma_0$ controls all the energies on the time slices $\Sigma\tau$ for $\tau\in [0,T]$. As a consequence, a convergence stated for $H^1(\Sigma_0)$ actually holds in $L^\infty([0,T], H^1(\Sigma_\tau))$.
\end{remark}

One now proceeds by extracting successively sequences of $(\lambda_n)$ as follows:
\begin{enumerate}[wide, labelwidth=!, labelindent=0pt]
\item Since $(\phi_n)$ is bounded in $C^1([0,T], L^2(\Sigma_\tau))\cap C^0([0,T], H^1(\Sigma_\tau))$, $(\phi_n)$ is bounded in $H^1(U)$. Hence, using Kakutani's theorem, there exists a sub-sequence of $(\phi_n)$ converging weakly in $H^1(U)$ towards a function $\phi$ in $H^1(U)$. Furthermore, using Rellich-Kondrachov's theorem, since $H^1(U)$ is compactly embedded in $H^\sigma(U)$ for $\sigma<1$, a diagonal extraction process gives:
\begin{eqnarray}
\phi_n&\stackrel{w - H^1(U)}{\xrightarrow{\hspace*{3cm}}}& \phi\label{hor5}\\
\phi_n&\stackrel{H^\sigma(U)}{\xrightarrow{\hspace*{3cm}}}& \phi\label{hor6}.
\end{eqnarray}
\item Since $(\phi_n)$ is bounded in $L^\infty([0,T], H^1(\Sigma_\tau))$, accordingly to Remark \ref{hor4}, up to an extraction, using Kakutani's theorem, one has :
\begin{eqnarray}
\phi_n&\stackrel{w - L^\infty([0,T],H^1(\Sigma_\tau))}{\xrightarrow{\hspace*{3cm}}}& \phi\label{hor7}
\end{eqnarray}
\item Furthermore, since $(\phi_n)$ converges strongly in $H^{\frac12}(U)$, by continuity of the trace operator from $H^{\frac12}(U)$ into $L^2(\Sigma_0)$ and using Remark \ref{hor4}, one gets that
\begin{equation}
\phi_n \stackrel{C^0([0,T], L^2(\Sigma_\tau))}{\xrightarrow{\hspace*{3cm}}} \phi\label{hor8}.
\end{equation}
We have also used that $\phi_n$ lies for all $n$ in $C^0([0,T], L^2(\Sigma_\tau))$ which is closed in $L^\infty([0,T], L^2(\Sigma_\tau))$.
\item Furthermore, $(\phi_k)$ and $(\partial_\tau \phi_k)$ are both bounded respectively in $L^\infty([0,T], H^1(\Sigma_\tau))$ and $L^\infty([0,T],L^2(\Sigma_\tau))$. Using Banach-Alaoglu-Bourbaki theorem, up to extractions, one has
\begin{eqnarray}
\phi_n&\stackrel{\star-w - L^\infty([0,T], H^1(\Sigma_\tau))}{\xrightarrow{\hspace*{3cm}}}& \phi\label{hor9}\\
\phi_n&\stackrel{\star-w-L^\infty([0,T],L^2(\Sigma_\tau))}{\xrightarrow{\hspace*{3cm}}}& \phi.\label{hor10}
\end{eqnarray}
\item Since the foliation $\{\Sigma_\tau\}$ has the volume of its leaves bounded away from $0$, Sobolev embeddings can be realized uniformly over the foliation. Using these Sobolev embeddings, Remark \ref{hor4} and the convergence \eqref{hor7}, the following convergence holds
\begin{equation}
\phi_n \stackrel{w-L^\infty([0,T], L^2\cap L^6(\Sigma_\tau))}{\xrightarrow{\hspace*{3cm}}} \phi\label{hor11}\\
\end{equation}
and, as a consequence,
\begin{equation}
\phi_n^3 \stackrel{w-L^\infty([0,T], L^2(\Sigma_\tau))}{\xrightarrow{\hspace*{3cm}}} \phi^3\label{hor12}.
\end{equation}
\item  Furthermore, since $g_\lambda$ is smooth, $\square_\lambda \phi_\lambda$ also converges in the sense of distributions towards $\square \phi$. Using the convergence \eqref{hor12}, $(\phi_n^3)$ converges towards $\phi^3$ in the sense of distributions. This means that the function $\phi$ satisfies, in the sense of distributions,
$$
\square \phi  = \phi^3.
$$
\item It remains to prove that $\phi$ satisfies the initial conditions.  This is a direct consequence of Equation \eqref{hor8}:
$$
\phi|_{C^-(p)} = \phi_0.
$$
\end{enumerate}

The next step of the proof consists in proving that the solution is in fact a strong solution of the equation:
\begin{enumerate}[wide, labelwidth=!, labelindent=0pt]
\item Since we are working on a compact space, the trace operator is in fact compact from $H^1(U)$ in $L^2(C^-(p))$. Up to an extraction, $(\phi_n)$ converges then strongly in $L^2(C^-(p))$. As a consequence, $\phi|_{C^-(p)}$ is equal to $\phi_0$ in $L^2(C^-(p))$.
\item One already has, as a consequence of the a priori estimates,
\begin{eqnarray*}
\phi &\in & L^\infty([0,T],H^1(\Sigma_\tau))\\
\phi &\in & C^0([0,T],L^2(\Sigma_\tau))\\
\partial_\tau\phi &\in & L^\infty([0,T],L^2(\Sigma_\tau)).
\end{eqnarray*}
\item One then considers the operator
$$
L_0=\square_\lambda - \frac{1}{N^2}\partial_\tau^2 = \alpha \partial_t\times L_1 + L_2.
$$
where $\alpha$ is a smooth function on $U$, $L_1$ is a first order purely spatial operator and $L_2$ is a second order purely spatial operator. It is clear from this decomposition that, if $\phi$ lies in $L^\infty([0,T], H^1(\Sigma_\tau))$ and $\partial_\tau \phi$ in $L^\infty([0,T], L^2(\Sigma_\tau))$, then $L_0\phi$ lies $L^\infty([0,T], H^{-1}(\Sigma_\tau))$ since $H^1(\Sigma_\tau)$ and $L^2(\Sigma_\tau)$ embed themselves continuously and uniformly over the foliation in $H^{-1}(\Sigma_\tau)$.
 \item One finally gets, using the principle of intermediate derivatives of Lions \cite{Lions:1963tu}, as described in Appendix \ref{tracelm}, since
\begin{eqnarray*}
 \phi&\in&L^\infty([0,T], H^1(\Sigma_\tau))\\
  \partial_\tau\phi &\in& L^\infty([0,T], L^2(\Sigma_\tau))\\
 \partial_\tau^2\phi &\in& L^\infty([0,T], H^{-1}(\Sigma_\tau)),
\end{eqnarray*}
and since the energy is continuous, one gets
\begin{eqnarray}
\phi &\in &C^0([0,T], L^2(\Sigma_\tau)) \nonumber\\
\partial_\tau \phi &\in & C^0([0,T], L^2(\Sigma_\tau)) \label{hor13}
\end{eqnarray}
that is to say that \(\phi \in C^1([0,T], L^2(\Sigma_\tau))\).
\item It remains to prove that the function $\phi$ belongs to $C^0([0,T], H^1(\Sigma_\tau))$. This is done as follows:
\begin{enumerate}[wide, labelwidth=!, labelindent=0pt]
\item The weak convergence \eqref{hor7} of $(\phi_n)$ in $L^\infty([0,T], H^1(\Sigma_\tau))$ and the strong convergence \eqref{hor8} of $(\phi_n)$ in $C^0([0,T], L^2(\Sigma_\tau))$ imply that
$$
\forall t\in[0,T], \phi_n(t, \star)\in H^1(\Sigma_\tau).
$$
\item The a priori estimates and Equation \eqref{hor13} give:
\begin{eqnarray*}
\phi &\in & L^\infty([0,T], H^1(\Sigma_\tau))\\
\partial_\tau \phi & \in & C^0([0,T], L^2(\Sigma_\tau))
\end{eqnarray*}
As a consequence (this fact is proved in \cite[p. 535]{MR2244222}), we have:
$$
\phi\in C^0([0,T], w-H^1(\Sigma_\tau)).
$$
In this context, on can prove the a priori estimates over the foliation $\Sigma_\tau$:
$$
|E_\phi(\Sigma_\tau)-E_\phi(\Sigma_\mu)|\leq C\left(\int_{[\tau;\mu]}E_\phi(\Sigma_r)\ud r\right).
$$
The energy is then locally Lipschitz continuous. As a consequence, since $\phi$ is in $C^0([0,T], w-H^1(\Sigma_\tau))$, $\phi$ is also in $C^0([0,T], H^1(\Sigma_\tau))$.
\end{enumerate}
\end{enumerate}
\end{proof}

\section{Estimates for the characteristic Cauchy problem}\label{sec:uniqueness}
In the previous section, it has been proved that the Cauchy problem admits a local solution to the characteristic Cauchy problem, relying only on the existence of a priori estimates for the solutions. The uniqueness has, so far, not be proved yet, as well as the continuity in the initial data. This can be achieved in various ways. The path we chose to follow is based on the work of Baez-Segal-Zhou \cite{MR1073286} and relies on an a priori existence result of solutions to the Cauchy problem, which has been handled in a specific way separately in section \ref{sec:car}. The estimates are established by relying on a reduction to the Cauchy problem.

The setting is both the framework introduced by H\"ormander and Baez-Segal-Zhou. Hence, one considers the following geometric context:
\begin{itemize}
\item Let $p$ a point in $M$.
\item One considers an open geodesically convex hyperbolic neighbourhood $\Omega$ of $p$.
\item Let $C^-(p)$ be the past light cone from $p$.
\item Let $t$ be a time function on $\Omega$; the time foliation arising from it is denoted by $\{S_t\}$. One assumes that $S_0\cap J^-(p)$ is entirely on $\Omega$. $S_T$ denotes the time slice containing $p$.
\item Let $T^a$ be a unit normal vector field to the foliation $S_t$ and considers the flow generated by $T$. The hypersurface $\Sigma_T$ is transported by the flow generated by $T^a$ up to $p$. This creates a foliation from $S_T$, which is denoted by $\{\Sigma_t\}$. Up to a rescaling and a shift, one can assume that the slice $\Sigma_0$ contains $p$.
\item The closure of the manifold which is obtained is now topologically isomorphic to $\Sigma_T\times[T, 0]$.
\end{itemize}
\begin{remark} The setting introduced by H\"ormander considers only space-times of the form $X\times \R$, where $X$ is a compact manifold, a priori without boundary. The result can nonetheless be immediately extended to the case when the manifold $X$ has a boundary. This can be seen using the following remark: the energy space remains $H^1(X)$ (and not $\dot{H}{}^1(X)$). Assume that the manifold with boundary $\overline{X}$ is embedded in a bigger compact manifold without boundary (which can always be done). If the boundary of $X$ is at least $C^1$, the space $H^1(X)$ extends continuously into $H^1(\overline{X})$. As a consequence, all energy estimates involving $\overline{X}$ can be brought back onto $X$. In this specific case, the boundary of $\Sigma_T$ is the intersection of $\Sigma_T$ with $J^-(p)$ and is, as a consequence, smooth.
\end{remark}

One introduces first the following operator:
$$
\mathfrak{T} : H^1(\Sigma_T)\times L^2(\Sigma_T) \rightarrow H^1(C_T)
$$
which associates to initial data $(\phi, \psi)$ the trace over $C_{T}$ of the solution of the \underline{linear} wave equation with initial data $(\phi, \psi)$ on $\Sigma_T$. Following H\"ormander \cite{MR1073287}, this operator as the following properties:
\begin{theorem}[H\"ormander] The linear operator $\mathfrak{T}$ is one-one, onto and bi-continuous, that is to say that there exists a constant $C$ depending only on the geometry of the manifold such that, for all $(\phi, \psi)$ in $H^1(\Sigma_T)\times L^2(\Sigma_T)$:
$$
\Vert\mathfrak{T}(\phi, \psi)\Vert_{H^1(C_T)} \lesssim \Vert(\phi, \psi)\Vert_{H^1(\Sigma_T)\times L^2(\Sigma_T)}
$$
and
$$
\Vert(\phi, \psi)\Vert_{H^1(\Sigma_T)\times L^2(\Sigma_T)} \lesssim  \Vert\mathfrak{T}(\phi, \psi)\Vert_{H^1(C_T)}.
$$
\end{theorem}

The remark of Baez, Segal and Zhou \cite[Proof of Theorems 13 and 16]{MR1073286} is then the following:
\begin{lemma}
Let $\delta_0$ be an initial data set in $H^1(C_T)$. Let $H$ be a function defined over the foliation $\Sigma_T$ in $C^1([0, T], L^2(\Sigma_t))\cap C^0([0, T], H^1(\Sigma_t))$. Then $\phi$ can be extended to the past of $\Sigma_T$ by solving the Cauchy problem on $\Sigma_T$ with initial data $\mathfrak{T}^{-1}(\delta_0)$ for the equation:
$$
\square \delta + \textbf{1}_{J^-(C_T)} H^2 \delta=0
$$
where the function $H$ is extended by 0 outside $J^-(C_T)$.
\end{lemma}

For such an equation, the energy estimates are simple to obtain, since the the foliation of reference to establish them have non vanishing volume. For the sake of consistency, these estimates are nonetheless proved here:
\begin{proposition}\label{charest} The following energy estimates hold for $\delta$: there exists an increasing function $C: \R^+\rightarrow \R^+$ such that:
$$
E_\delta(C_T)\leq C (\sup_{[0,T]} E^2_H(\Sigma_t)) E_\delta(\Sigma_T)
$$
and
$$
 E_\delta(\Sigma_T) \leq C \sup_{[0,T]} E^2_H(\Sigma_t))E_\delta(C_T),
$$
where:
\begin{itemize}
 \item the energy on $C_T$ is taken to be the squared $H^1$-norm of the function on $C_T$:
$$
E_\delta(C_T) = \Vert \delta\Vert^2_{H^1(C_T)}
$$
\item the energy on the time slice $\Sigma_T$ is taken to be the standard energy:
$$
E_\delta(\Sigma_T) = \Vert \delta\Vert^2_{H^1(\Sigma_T)}+ \Vert T^a\nabla_a \delta\Vert^2_{L^2(\Sigma_T)}.
$$
\end{itemize}

\end{proposition}
\begin{proof} Since the techniques under considerations are standard, the proof is only sketched. Consider the stress energy tensor:
$$
T_{ab}=\nabla_a \delta \nabla_b \delta +g_{ab}\left(-\frac12 \nabla_c\delta\nabla^c\delta + \frac{\delta^2}{2}\right)
$$ whose error term is given by:
$$
\nabla^a(T^bT_{ab}) = \nabla^{(a}T^{b)}T_{ab} - T^a\nabla_a\delta H^2 \delta.
$$
Applying Stokes theorem between $\Sigma_T$ and $\Sigma_s$, one gets immediately the following inequality:
$$
\left|E_{\delta}(\Sigma_T) - E_\delta(\Sigma_s)\right| \lesssim \int_s^TE_{\delta}(\Sigma_\tau) \ud \tau + \int_{s}^T\int_{\Sigma_\tau} H^4\delta^2 \ud\mu_{\Sigma_\tau}\ud \tau,
$$
One deals with the non linear term using Sobolev embeddings over the compact foliation $\Sigma_\tau$, whose volume of leaves does not go to 0:
\begin{gather*}
\int_{s}^T\int_{\Sigma_\tau} H^4\delta^2 \ud\mu_{\Sigma_\tau}\ud \tau \lesssim \int_{s}^T\left(\int_{\Sigma_\tau} H^6 \ud\mu_{\Sigma_\tau} \right)^{2/3}\left(\int_{\Sigma_\tau} \delta^6\ud\mu_{\Sigma_\tau}\right)^{1/3}\ud \tau\\
\lesssim\sup_{\tau\in[0,T]}E^2_H(\Sigma_\tau)\int_{s}^TE_{\delta}(\Sigma_\tau) \ud \tau
\end{gather*}
One finally gets
$$
\left|E_{\delta}(\Sigma_t) - E_\delta(\Sigma_s)\right| \lesssim \left(\sup_{\tau\in[0,T]} E^2_H(\Sigma_\tau)+1\right) \int_{s}^TE_{\delta}(\Sigma_\tau) \ud \tau
$$
and using Gr\"onwall's inequality closes the estimates.

\end{proof}

Let now be $\theta$ and $\xi$ be two functions in $H^1(C_T)$ and consider the two characteristic Cauchy problems:
$$\left\{
\begin{array}{l}
 \square u +\frac16  \text{scal}_g u = u^3\\
u|_{C_T}=\theta
\end{array}\right.
\text{ and }
\left\{
\begin{array}{l}
 \square v +\frac16  \text{scal}_g v = v^3\\
v|_{C_T}=\xi
\end{array}\right.
$$
The function $\delta$ defined as being the difference between $u$ and $v$ satisfies the wave equation
$$
\square \delta + \frac16 \text{scal}_g \delta = H^2\delta 
$$
where
$$
H^2 = \left(u^2+v^2+uv \right) = \left(u+\frac{v}{2}\right)^2+\frac34v^2.
$$
A direct consequence of Proposition \ref{charest} is the following energy estimates:
\begin{proposition}\label{charest2}  There exists an increasing function $C$ such that the following inequalities hold:
\begin{gather*}
\Vert u-v \Vert^2_{H^1(\Sigma_T)} +\Vert T^a\nabla_a(u-v)\Vert^2_{L^2(\Sigma_T)} \leq C\left(\Vert u\Vert^2_{H^1(C_T)}+\Vert v\Vert^2_{H^1(C_T)} \right)\cdot ||\theta-\xi\Vert^2_{H^1(C_T)}\\
 ||\theta-\xi\Vert^2_{H^1(C_T)}\leq C\left(\Vert u\Vert^2_{H^1(C_T)}+\Vert v\Vert^2_{H^1(C_T)}\right) \cdot \left(\Vert u-v \Vert^2_{H^1(\Sigma_T)} +\Vert T^a\nabla_a(u-v)\Vert^2_{L^2(\Sigma_T)}\right)
\end{gather*}
\end{proposition}
\begin{proof} The proof of this energy inequalities is a direct consequence of Proposition \ref{charest}. It suffices to notice that, for $H^2=u^2 +uv + v^2$, using the triangular inequality,
\begin{eqnarray*}
\sup_{[T,0]} E_H(\Sigma_t)^2 \lesssim \left(\sup_{[T,0]} E_u(\Sigma_t)^2 +\sup_{[T,0]} E_v(\Sigma_t)^2\right).
\end{eqnarray*}
Using a priori estimates such as the one proved in \cite[Proposition 6.2]{Joudioux:2012eo}, the later terms can be bounded by either
$$
E_u(\Sigma_T)^2+ E_v(\Sigma_T)^2
\text{ or }
E_u(C_T)^2+ E_v(C_T)^2.
$$
\end{proof}

Finally, an important consequence of the energy estimate \eqref{charest2} and Theorem \ref{thm:existence} is the following proposition:
\begin{proposition}\label{uniqueness} The characteristic Cauchy problem:
 $$\left\{
\begin{array}{l}
 \square u +\frac16 \scal_g u = u^3\\
u|_{C_T}=\theta \in H^1(C_T)
\end{array}\right.
$$
 admits at most local solution up to time $T$ in $C^0([T, 0],H^1(\Sigma_t))\cap C^1([T,0], L^2(\Sigma_t))$.
\end{proposition}

\section{Solving the characteristic Cauchy problem in the neighbourhood of $i^0$}\label{sec:neighbourhoodi0}

\subsection{Preliminary result}

The purpose of this section is to explain how the characteristic Cauchy problem can be solved in the neighbourhood of the spacelike infinity $i^0$ of the asymptotically simple manifold. The neighbourhood of $i^0$ is assumed to be isometric to the Schwarzschild space-time to agree with the work of the Corvino-Schoen and Chrusciel-Delay.

The Schwarzschild metric writes, in the standard spherical coordinates
$$
\hat{g}_S = \left(1-\frac{2m}{r}\right) \ud t^2-  \left(1-\frac{2m}{r}\right)^{-1} \ud r^2 -r^2d\omega_{\mathbb{S}^2}.
$$
Performing the change of coordinates
$$
r^\ast = r+ 2m \log \left(1-\frac{2m}{r}\right), R = \frac{1}{R} \text{ and } u = t-r^\ast,
$$
the metric takes the form
$$
\hat{g}_S =(1-2mR)\ud u^2 -\frac{2}{R^2} \ud u \ud R -\frac{1}{R^2} \ud \omega_{\mathbb{S}^2}.
$$
A conformal rescaling with a conformal factor $\Omega =\frac{1}{R}$ is finally performed to define the unphysical metric:
$$
 g_S =R^2(1-2mR)\ud u^2 - 2 \ud u \ud R -\frac{1}{R^2} \ud \omega_{\mathbb{S}^2}.
$$
Consider the domain $\Omega^+_{u_0} = \{u\leq u_0\}$, for some $u_0$ to be chosen later. It has been proved \cite{Joudioux:2012eo, mn04} that the following lemma holds:
\begin{lemma}\label{schwarzestimates} Let $\epsilon>0$. There exists $u_0<0$,  $|u_0|$ large enough, such that the following decay estimates in the coordinate $(u,r,\theta, \psi)$ hold:
\begin{gather*}
r<r^\ast<(1+\epsilon)r,
1<Rr^\ast<1+\epsilon,
0<R|u|<1+\epsilon,\\
1-\epsilon<1-2mR<1,
0<s=\frac{|u|}{r^\ast}<1.
\end{gather*}
Furthermore, the vector field
$$
 T^a = u^2\partial_u-2(1+uR)\partial_R
$$
is uniformly timelike in the region $\Omega^+_{u_0}$.
\end{lemma}
The parameter $\epsilon$ will be chosen later when we perform the energy estimates. We define, in $\Omega^+_{u_0}=\{t>0, u<u_0\}$, the following hypersurfaces, for $u_0$ given in $\R$:
\begin{itemize}
\item $S_{u_0}=\{u=u_0\}$, a null hypersurface transverse to $\scri^+$;
\item $\Sigma_{0}^{u_0>}=\Sigma_0\cap\{u_0>u\}$, the part of the initial data surface $\Sigma_0$ in the past of $S_{u_0}$;
\item $\scri_{u_0}^+=\Omega^+_{u_0}\cap \scri^+$, the part of $\scri^+$  beyond $S_{u_0}$;
\item $\mathcal{H}_s=\Omega^+_{u_0}\cap\{u=-sr^\ast\}$, for $s$ in $[0,1]$, a foliation of $\Omega^+_{u_0}$ by spacelike hypersurfaces accumulating on $\scri$.
\end{itemize}

The volume form associated with $ g$ in the coordinates $(R,u,\omega_{\mathbb{S}^2})$ is then:
\begin{equation}\label{volumeformcoor}
\mu[ g]= \ud u\wedge\ud  R \wedge \ud^2\omega_{\mathbb{S}^2}.
\end{equation}

If $\phi$ is function defined on $\Omega^+_{u_0}$, one defines the following energies, in the domain $\Omega_{u_0}^+$:
\begin{proposition}\label{energyequivalenceschwarzschild}
There exists $u_0$, such that the following energy estimates holds on $\mathcal{H}_s$ in $\Omega^+_{u_0}$, for all $s$ in $[0,1]$:
$$
E_{\phi}(\mathcal{H}_s) = \int_{\mathcal{H}_s}i^\star_{\mathcal{H}_s}\left(\star \hat T^a T_{ab}\right)\approx \int_{\mathcal{H}_s}\left(u^2(\partial_u\phi)^2+\frac{R}{|u|}(\partial_R\phi)^2+|\nabla_{\mathbb{S}^2}\phi|^2+\frac{\phi^2}{2}+\frac{\phi^4}{4}\right)\ud u \wedge \ud \omega_{\mathbb{S}^2}.
$$
Furthermore, if one introduces the parameter
\begin{equation}\label{parametrization}
\tau:
\begin{array}{ccc}
[0,1]&\longrightarrow& [0,2]\\
s&\longmapsto& -2(\sqrt{s}-1),
\end{array}
\end{equation} the following energy estimates holds:
\begin{itemize}
 \item Energy decay:
$$
 E_{\phi}(\mathcal{H}_s)     \lesssim E_{\phi}(\som)
$$

\item A priori estimates:
$$
E(\scri^+_{u_0})+\int_{\scri^+_{u_0}}\phi^4\ud \mu_{\scri^+}+E(S_{u_0})+ \int_{S_{u_0}}\phi^4\ud \mu_{S_u} \approx E(\som) + \int_{\som} u^4 \ud_{\som}
$$
\end{itemize}
\end{proposition}

\subsection{Local existence for the Characteristic Cauchy problem on $\scri^+$}

The purpose of this section is to prove that the Cauchy problem:
\begin{equation}\label{eq:carschwarzschild}
\begin{array}{l}
 \square \phi + \frac16\text{Scal}_{ g} \phi =  \phi^3 \\
 \phi|_{\scri^+_{u_0}} = \theta^+ \in H^1(\scri^+_{u_0}) \text{ and }  \phi|_{S_{u_0}} = \theta^0 \in H^{1}(S_{u_0})
\end{array}
\end{equation}
where $H^{1}(S_{u_0})$ is defined as being the completion of the space of traces of smooth functions on $S_{u_0}$ for the norm
$$
 \int_{S_{u_0}} i^\star_{S_{u_0}}\left(\star \hat T^a T_{\text{lin}\,ab}\right)
$$
$T_{\text{lin}\,ab}$ being the stress-energy tensor associated with the linear wave equation.
In a previous work \cite{Joudioux:2012eo}, the following result has been proved:
\begin{itemize}
\item using the fact that the leaves of the foliation $\mathcal{H}_s$ endowed with its induced metric is uniformly equivalent to a cylinder, it is possible to establish energy estimates for the difference of solutions of the non linear equation;
\item it is possible to establish an existence result for solutions of the nonlinear equation for small data, using either a Picard iteration (as in \cite{Joudioux:2010tn}) or H\"ormander slowing down process.
\end{itemize}
\begin{proposition}\label{prop:schwarz1} Let $\phi$ and $\psi$ be two solutions of the characteristic Cauchy problem and denote by $\delta = u-v$ their difference. One denotes the energy of $\delta$ by
$$
E_{\delta}(\mathcal{H}_s) = \int_{\mathcal{H}_s}i^\star_{\mathcal{H}_s}\left(\star \hat T^a T_{\text{lin}\,ab}\right)\approx \int_{\mathcal{H}_s}\left(u^2(\partial_u\delta)^2+\frac{R}{|u|}(\partial_R\delta)^2+|\nabla_{\mathbb{S}^2}\phi|^2+\frac{\delta^2}{2} \right)\ud u \wedge \ud \omega_{\mathbb{S}^2}
$$
where $T_{\text{lin}\,ab}$ is the stress energy tensor for $\delta$ associated with the linear wave equation. Then, $\delta$ satisfies the following energy estimates on the foliation $\mathcal{H}_s$;
$$
E_\delta(\mathcal{H_s}) \lesssim  \sup_{\sigma\in [s;t]} \left( ||\phi||^4_{H^1(\mathcal{H}_s)} + ||\psi||^4_{H^1(\mathcal{H}_s)}\right) E_\delta(\mathcal{H_s}).
$$
In particular, the Cauchy problem \eqref{eq:carschwarzschild} admits at most one solution in $C^0([0,\epsilon], H^1(\mathcal{H}_s))\cap C^1([0,\epsilon], L^2(\mathcal{H}_s))$ and the energy
$$
s\mapsto  E_\delta(\mathcal{H}_s)
$$
is continuous.
\end{proposition}

\begin{proof} Let $\phi$ and $\psi$ be two global solutions of \eqref{eq:carschwarzschild}. Their difference $\delta =\phi -\psi$ satisfies the equation:
$$
\square \delta + \frac{1}{6}\text{scal}_{\hat g} \delta = \left(\phi^2+\phi\psi+\phi^2\right)\delta.
$$

Using the Stokes theorem between $\scri^+_{u_0}$, $S_{u_0}$ and $\mathcal{H}_s$, one gets, since $\delta$ vanishes on $S_{u_0}$ and $\scri^+_{u_0}$, using proposition:
\begin{gather*}
E_{\delta}(\mathcal{H}_s) \lesssim 
 \int_{0}^\varepsilon E_\delta(\mathcal{H}_\sigma) \ud \sigma +  \int_{0}^{s}\int_{\mathcal{H_{\sigma}}} \vert \hat T^a\hat \nabla_{a} \delta \cdot \delta\vert + (\phi^2+\psi^2)\vert\hat T^a\hat \nabla_{a} \delta \cdot \delta \vert \ud \mu_{\mathcal{H}_{\sigma}} \ud \sigma\\ + E_{\delta } (S_{u_0}) .
\end{gather*}
As already said, Sobolev estimates can be performed uniformly on the foliation $(\mathcal{H}_s)$. As a consequence, using H\"older's inequality, one gets:
$$
\int_{0}^{s}\int_{\mathcal{H_{\sigma}}}(\phi^2+\psi^2)\vert \hat T^a\hat \nabla_{a} \delta \cdot \delta \vert \ud \mu_{\mathcal{H}_{\sigma}}\leq  C\int_{0}^\varepsilon  E_\delta(\mathcal{H}_\sigma) \ud \sigma + \varepsilon \sup_{\sigma\in[0,\epsilon]} \left(E_{\phi}(\mathcal{H}_\sigma)^2+E_{\psi}(\mathcal{H}_\sigma)^2\right)
$$
where $T^a$ is the approximate Morawetz vector field
$$
T^a = u \partial_u + v \partial_ v.
$$
The constant depends only on the $L^\infty$-bound of the scalar curvature. Finally, using Gr\"onwall lemma, one gets that
$$
E_{\delta}(\mathcal{H}_s) \lesssim \varepsilon e^{C\varepsilon} \sup_{\sigma\in[0,\epsilon]} \left(E_{\phi}(\mathcal{H}_\sigma)^2+E_{\psi}(\mathcal{H}_\sigma)^2\right) E_{\delta}(\mathcal{H}_0) \times E_\delta (S_{u_0}).
$$
As a consequence, the mapping
\[
\begin{array}{ccc}
B(0,r)\subset H^1(\scri^+_{u_0}) \times  H^1(S_{u_0}) & \longrightarrow & L^\infty ([0,\epsilon], H^1(\mathcal{H}_s))\\
(\theta^+, \theta ^0)  & \longmapsto & \phi
\end{array}
\]
is continuous. Furthermore, the Cauchy problem \eqref{eq:carschwarzschild} admits at most a unique solution in $C^0([0,\epsilon], H^1(\mathcal{H}_s))\cap C^1([0,\epsilon], L^2(\mathcal{H}_s))$.
\end{proof}

\begin{proposition}\label{prop:existenceschwarz} Let $\varepsilon$ be a positive real number. Then, for $\varepsilon$ small enough, the characteristic Cauchy problem \eqref{eq:carschwarzschild} admits a unique solution in the neighborhood of $\scri^+_{u_0}$
$$
\bigcup_{s\in[0, \varepsilon]} H_{s} = \left\{ u \geq (1+\varepsilon) r^\star \right\}.
$$
\end{proposition}
\begin{proof} The proof of the proposition is based on a Picard iteration defined as follows; let $\phi_0$ be a solution of the linear Cauchy problem:
\begin{equation}\label{initialstep}
\begin{array}{l}
 \square \phi_0 +\text{scal}_{ g} \phi_0 = 0 \\
\phi_0|_{\scri^+_{u_0}} = \theta \in H^1(\scri^+_{u_0}).
\end{array}
\end{equation}
and define the sequence as follows, for $n>0$:
\begin{equation}\label{recursion}
\begin{array}{l}
 \square \phi_n +\text{scal}_{ g} \phi_n = \phi^3_{n-1} \\
\phi_n|_{\scri^+_{u_0}} = \theta \in H^1(\scri^+_{u_0}).
\end{array}
\end{equation}
H\"ormander's theorem \cite{MR1073287} ensures that these two Cauchy problems admits a global solution in $L^\infty ([0,1], H^1(\mathcal{H}_s))$. Using the stress energy tensor associated with the linear wave equation, one obtains the following energy estimates:
\begin{equation}
E_{\phi_n}(\mathcal{H}_s) \leq C \int_{0}^s  E_{\phi_n}(\mathcal{H}_\sigma) \ud \sigma + \int_{0}^s \int_{\mathcal{H}_s}\phi_{n-1}^6 \ud \mu{\mathcal{H}_\sigma}\ud \sigma,
\end{equation}
and, using Gr\"onwall's lemma and the Sobolev embedding on the foliation $\mathcal{H}_s$, one gets immediately that
\begin{equation}
\sup_{\sigma\in [0,s]}E_{\phi_n}  (\mathcal{H}_\sigma) \leq C s E_{\theta}(\scri^+_{u_0})\left(\sup_{\sigma\in [0,s]} E_{\phi_{n-1}}(\mathcal{H}_{\sigma}) \right)^3.
 \end{equation}
An immediate recursion gives that, for all $n>0$ and $s=\epsilon$
\begin{eqnarray*}
\sup_{\sigma\in [0,\epsilon]}E_{\phi_n}  (\mathcal{H}_\sigma) & \leq & \left( C\epsilon E_{\theta}(\scri^+_{u_0}) \left(\sup_{\sigma\in [0,\epsilon]} E_{\phi_{0}}(\mathcal{H}_{\sigma}) \right)^\frac12  \right)^{3^n}\\
&\leq&\left( C\epsilon \left( E_{\theta}(\scri^+_{u_0}\right)^{\frac12} \right)^{3^n}.
\end{eqnarray*}
As a consequence, the sequence $(\phi_n)$ is a Cauchy sequence in the complete space $L^\infty\left([0,\epsilon], H^1(\mathcal{H}_s) \right)$ and, as a consequence, converges strongly towards a function $\phi$ in the same space. Since $\phi_n$ is a weak solution of the Cauchy problem \ref{eq:carschwarzschild}, so is $\phi$. As a consequence, the previous proposition \ref{prop:schwarz1} proves that the solution actually belongs to $C^0\left([0,\epsilon], H^1(\mathcal{H}_s) \right)$. This concludes the proof of the existence of solution to \eqref{eq:carschwarzschild} in  $C^0\left([0,\epsilon], H^1(\mathcal{H}_s) \right)$.
\end{proof}

\section{Global Cauchy problem on $\scri^+$}\label{sec:globalcauchy}

The existence of global solutions to the Cauchy problem is proved using a gluing process: the Cauchy problem with data in $H^1(\scri^+)$ are constructed up to a spacelike hypersurface and, then, considering the traces of the solution on two give hypersurfaces in the neighbourhood of $i^0$ solved up to the initial time slice $\Sigma_0$. The main issue arising in this process is that the constants arising in the energy estimates depend on the $L^\infty$-bounds of the metric and its inverse on the bounded neighbourhood of $i^0$. Since the compactification we are working with has the particularity to have an asymptotic end at $i^0$, these constants are not bounded on the whole future of $\Sigma_0$. This problem is avoided as follows:
\begin{itemize}
\item Let $u_0$ be in $\mathbb{R}$ such as in Proposition \ref{energyequivalenceschwarzschild}.
 \item Let $\Sigma$ be a given spacelike hypersurface such that:
\begin{itemize}
\item $\Sigma$ is transverse to $\scri^+$;
\item $\Sigma$  is in the past of $S_{u_0}\cap J^+(\{t = 0\})$ and in the future of $\Sigma_{0}$; in particular, $\Sigma$ coincides with $\Sigma_{0}$ for from $i^0$.
\end{itemize}
\end{itemize}
These choices of $u_0$ and $\Sigma$ are made once for all.

\begin{proposition}\label{prop:global} The characteristic Cauchy problem
$$
\begin{array}{l}
 \square \phi + \frac{1}{6}\text{Scal}_{{g}} \phi = \phi ^3 \\
\phi|_{\scri^+} = \theta \in H^1(\scri^+).
\end{array}
$$
admits a unique global solution in \({C^0([0,1], H^1(\tilde \Sigma_s)) \cap C^1([0,1],L^2(\tilde \Sigma_s))}\)
where $(\tilde \Sigma_s)_{s>0}$ is a smooth spacelike foliation extending the definition of the foliation $(\mathcal{H}_s)$ in the neighbourhood of $i_0$.  Furthermore, if $\tilde \phi$ is another solution with initial data $\tilde \theta$, the following energy estimates holds: there exists a increasing function $f$ such that
$$
E_{\phi-\tilde \phi}(\Sigma_0) \leq  C f\left( \sup_{\sigma \in [0,1]} \left(||\phi||^2_{H^1(\mathcal{H}_s)} + ||\tilde \phi||^2_{H^1(\mathcal{H}_s}\right)  \right) ||\theta -\tilde \theta ||^2_{H^1(\scri^+)}.
$$
\end{proposition}
\begin{proof} The proof of the theorem relies on the consecutive use of Theorem \ref{thm:existence}, Proposition \ref{uniqueness} and of Proposition \ref{prop:existenceschwarz}.

Let $\theta$ be in $H^1(\scri^+)$. Using Theorem \ref{thm:existence}, one proves that the Cauchy problem admits a solution $\phi$ up to $\Sigma$, as defined earlier. Then, using Proposition \ref{prop:existenceschwarz}, there exists a $\epsilon>0$ such that the Cauchy problem given by
\begin{equation}\label{intermediate1}
\begin{array}{l}
 \square  \phi + \frac16\text{Scal}_{ g}  \phi = \phi^3 \\
 \phi|_{\scri^+_{u_0}} = \theta \in H^1(\scri^+_{u_0}) \text{ and }  \phi|_{S_{u_0}} = \phi_{S_{u_0}} \in H^{1}(S_{u_0})
\end{array}
\end{equation}
 exists in
$$
\bigcup_{s\in[0, \varepsilon]} H_{s} = \left\{ u \geq (1+\varepsilon) r^\star \right\}.
$$
Consider the time slice $\mathcal{H}_\epsilon$. This time slice can be extended into $\hat M$ into a spacelike Cauchy slice denoted by $\mathcal{H}$. The timelike vector field $T$ is extended as a timelike vector to $\mathcal{H}$.  One now considers the functions $(\psi_0,\psi_1)$ defined piecewise by:
\begin{equation*}
 \left\{\begin{array}{lcl}
  \psi_0 |_{\mathcal{H_\epsilon}} &=& \tilde{\phi} \in H^1(\mathcal{H}_\epsilon)\\
 \psi_0 |_{\mathcal{H}\cap J^+(S_{u_0})} &=& \tilde{\phi}\in H^1({\mathcal{H}\cap J^+(S_{u_0})}) \\
 \end{array}
\right.
\end{equation*}
and
\begin{equation*}
 \left\{\begin{array}{lcl}
  \psi_1 |_{\mathcal{H_\epsilon}} &=& T^a \nabla_a \tilde{\phi}\\
 \psi_1 |_{\mathcal{H}\cap J^+(S_{u_0})} &=& T^a \nabla_a \tilde{\phi}\in L^2({\mathcal{H}\cap J^+(S_{u_0})}) \\
 \end{array}
\right..
\end{equation*}
By construction, since $\mathcal{H}$ is a spacelike Cauchy surface on the unphysical space-time, the result from Cagnac-Choquet-Bruhat \cite{MR789558} can be applied immediately to prove the existence of a unique solution to the Cauchy problem with data $(\psi_0, \psi_1)$ on $\mathcal{H}$ up to the Cauchy surface $\Sigma_0$. Furthermore, this solution belongs to \({C^0([0,1], H^1(\tilde \Sigma_s)) \cap C^1([0,1],L^2(\tilde \Sigma_s))}\).

Consequently, there exists a global solution of the Cauchy problem:
$$
\begin{array}{l}
 \square \phi + \frac{1}{6}\text{Scal}_{{g}} \phi = \phi ^3 \\
\phi_{\scri^+} = \theta \in H^1(\scri^+).
\end{array}
$$
obtained by gluing the solutions of the Cauchy problems obtained in Theorem \ref{thm:existence} and Proposition \ref{prop:existenceschwarz}.
\end{proof}

\section{Existence and regularity of the scattering operator}\label{sec:scattering}

The purpose of this section is to prove the regularity and the existence of conformal scattering operator for the non-linear wave equation
$$
\square \hat u  = \hat u^3
$$
In the context of asymptotically simple space-times with regular $i^\pm$, such as the one considered in \cite{Joudioux:2012eo,mn04}, it has been proved by Mason and Nicolas, for the scalar wave equation, that the trace operators
$$
\mathfrak{T}^\pm: (u,T(u))|_{t=0} \in H^1(\Sigma_0)\times L^2(\Sigma_0) \longmapsto \phi|_{\scri^\pm} \in H^1(\scri^+)
$$
can be obtained from the inverse wave operators $\tilde\Omega^\pm$ as follows: let $F_\pm$ be the null geodesic flows identifying the hypersurface $\Sigma_0$ with $\scri^\pm$; then, $\mathfrak{T}^\pm$ are given by
$$
\mathfrak{T}^\pm = F^\star_\pm \tilde \Omega^\pm
$$
where $F^\star$ is the pullback by the null geodesic flows $F_\pm$. The same result will hold in the context of the nonlinear equation. The conformal scattering result which is stated here should consequently be considered as a standard scattering result for a non-stationary metric for a non-linear wave equation. The existence of scattering operators was obtained on the flat background for sub-critical wave equations by \cite{Tsutaya:ep,Tsutaya:1992vx,Karageorgis:2006wg,hi03,Hidano:1998da}.

One considers the operators
$$
\mathfrak{T}^\pm: H^1(\Sigma_0)\times L^2(\Sigma_0) \rightarrow H^1(\scri^\pm)
$$
defined by
$$
\mathfrak{T}^\pm(\phi, \psi) = u|_{\scri^\pm}
$$
where $u$ is the unique solution of the Cauchy problem on the conformally compactified space-time
$$
\left\{
\begin{array}{c}
\square u +\frac16 \text{Scal}_{ g}u = u^3 \\
u|_{t=0} = \phi, \hat T^a\nabla_a u|_{t=0} = \psi.
\end{array}
\right.
$$
One finally introduces the conformal scattering operator defined as
$$
 \mathfrak{S} = \mathfrak{T}^+ \circ \left(\mathfrak{T}^-\right)^{-1}
$$

\begin{theorem} \label{thm:mainthm} The operators $\mathfrak{T}^\pm : H^1(\Sigma_0)\times L^2(\Sigma_0) \rightarrow H^1(\scri^\pm)$ are well-defined, invertible and locally bi-Lipschitz, that is to say that $\mathfrak{T}^\pm$ and $(\mathfrak{T}^\pm)^{-1}$ are Lipschitz on any ball in $H^1(\Sigma_0)\times L^2(\Sigma_0) $ and $H^1(\scri^\pm)$ respectively.
\end{theorem}
\begin{proof} The existence and continuity of the operators is guaranteed the existence of solution to the Cauchy problem for the conformal wave equation, see Proposition \ref{hor3}. The existence of their inverses, and their continuity, are obtained thanks to the Lipschitzness guaranteed by the energy estimate of Proposition \ref{prop:global}. 
\end{proof}

The final product of this paper, the existence of a scattering operator, is achieved by composing the operator \((\mathfrak{T}^{-})^{-1}\) with \(\mathfrak{T}^{+}\):

\begin{corollary}\label{cor:existencescatt} There exists a locally bi-Lipschitz invertible operator \(\mathfrak{S}  : H^{1}(\Sigma) \times L^{2}(\Sigma) \rightarrow H^1(\scri^\pm) \). This operator generalises the notion of classical scattering operator in the situation when the metric is static.
\end{corollary}
\begin{proof} As mentioned before, the construction of \(\mathfrak{S} \) is a straightforward consequence of Theorem \ref{thm:mainthm}. The fact that the operator coincides with the classical notion of scattering operator is stated in \cite{mn04}.
\end{proof}

\appendix
\section{A trace theorem}\label{app:trace}

One of the critical point of the proof of the existence theorem \ref{hor3} is a trace theorem, which has already been used in \cite[p. 19, proof of Theorem 4.3]{MR2244222}. Most of the material presented here can be found in \cite[Section 4.3, p. 281 sqq.]{MR1395148} for the interpolation and in \cite[Theorems 2.3 and 3.1]{Lions:1972ww} for the continuity of trace operators. Finally, it is important to note that the essential ideas and proofs of this appendix are contained in \cite[Chapter 3, Section 8.2]{Lions:1972ww}. The purpose of this appendix is also to give further understanding, and details on the work of H\"ormander on the Cauchy problem \cite{MR1073287} and clarify some points contained in \cite{MR2244222}.

Let $X$ and $Y$ be two Hilbert spaces such that
\begin{itemize}
 \item $X$ is continuously embedded in $Y$;
\item $X$ is dense in $Y$.
\end{itemize}
The interpolation spaces between $X$ and $Y$ are denoted by
$$
[X,Y]_{\theta} \text{ for } \theta \in [0,1].
$$

Let $a,b$ be two elements in $\mathbb{R}\cup\{\pm \infty\}$ and $m$ an positive integer. One considers the following space:
$$
W(a,b) =\left\{u |u \in L^2([a,b],X), \frac{\partial^m u}{\partial t^m} \in L^2([a,b],Y)\right\}.
$$
The following trace theorem then holds, see \cite[theorems 2.3 and 3.1, chapter 1,]{Lions:1972ww}:
\begin{theorem}[Lions-Magenes]\label{tracelm} Let $u$ be in $W(a,b)$. Then, for all $j$ in $\{0,\dots,m-1 \}$, the derivatives of $u$ satisfy:
\begin{itemize}
 \item $\frac{\partial^j u}{\partial t^j}$ is in $L^2([a,b];[X,Y]_{\frac{j}{m}})$;
\item $\frac{\partial^j u}{\partial t^j}$ is in $C^0_b([a,b];[X,Y]_{\frac{j+1/2}{m}})$.
\end{itemize}
\end{theorem}
\begin{remark} The theory developed by Lions and Magenes applies to general Hilbert spaces, independently of the geometric context (in particular, if one works on a space with boundary or on a compact manifold).
\end{remark}

One applies this result to the following situation: let $M$ be a compact manifold and consider
\begin{itemize}
\item $X= H^1(M)$;
\item $Y = H^{-1}(M)$.
\end{itemize}
The interpolation spaces between both are given by \cite[Section 4.3, Proposition 3.1]{MR1395148}:
$$
[H^1(M),H^{-1}(M)]_\theta = H^{1-2\theta}(M) \text{ for } \theta \in [0,1].
$$
Let $a,b$ be in $\mathbb{R}$ and $m=2$. Applying Theorem \ref{tracelm} then gives: if u satisfies:
$$
u\in L^2([a,b],H^1(M)) \text{ and }\frac{\partial^2 u}{\partial t^2} \in L^2([a,b],H^{-1}(M))
$$
then:
\begin{itemize}
 \item $\frac{\partial u}{\partial t}$ is in $L^2([a,b],L^2(M))$ ($\theta = \frac12$);
 \item $u\in C^0([a,b],L^2(M))$;
\item $\frac{\partial u}{\partial t}$ is in $C^0([a,b],H^{-\frac12}(M))$ ($\theta = \frac34$).
\end{itemize}

\begin{remark}
Obtaining that the derivative is actually in $C^0([a,b],H^{0}(M))$ would require
$$
u\in L^2([a,b],H^{1}(M)) \text{ and }\frac{\partial^2 u}{\partial t^2} \in L^2([a,b],H^{-\frac14}(M))
$$
\end{remark}

To improve this result, one notices the following:
\begin{lemma} Let $I=[a,b]$ be a time interval and $s<\sigma $ be two real numbers and consider a function $\phi$ such that
$$
\phi \in L^\infty(I,H^s(M)) \cap C^0(I,H^\sigma (M)).
$$
 Then, $\phi$ belongs to  $C^0(I,H^s (M)-w)$.
\end{lemma}
\begin{remark} This is a particular case of  \cite[Chapter 3, Lemma 8.1]{Lions:1972ww}. 
\end{remark}

\begin{proof} Let $\phi$ be in $H^s(M)$ such as in the proposition and consider $w$ in $H^\sigma(M)$. One considers the function
$$
f(t) = <\phi(t),w>_{H^s}.
$$
Let $t_0$ be in $I$. The purpose is to prove that $f$ is continuous in $t_0$.

Let $(w_k)$ be a sequence of smooth function converging towards $w$ in $H^s$. Let $\epsilon$ be a positive real and consider:
\begin{eqnarray*}
 f(t)-f(t_0) &=& <\phi(t)-\phi(t_0),w>_{H^s}\\
&=&<\phi(t)-\phi(t_0),w-w_k>_{H^s} +<\phi(t)-\phi(t_0),w_k>_{H^s}.
\end{eqnarray*}
Since, using Cauchy-Schwarz inequality,
$$
|<\phi(t)-\phi(t_0),w-w_k>_{H^s}|\leq \Vert w-w_k\Vert_{H^s}\Vert\phi(t)-\phi(t_0)\Vert_{H^s},
$$
and since $(w_k)$ converges towards $w$ in $H^s$, and since $\phi$ is in $L^\infty(I,H^s(M))$, there exists a $K$ such that:
$$
\left|<\phi(t)-\phi(t_0),w-w_K>_{H^s}\right|\leq \frac{\epsilon}{2}.
$$

One now considers the interpolation operators between $L^2(M)$ and $H^s(M)$ (for arbitrary $s$):
$$
\mathfrak{F}_s=L^2(M) \longrightarrow H^s(M).
$$
This is a family of unbounded, self-adjoint operators for which:
$$
\Vert \mathfrak{F}^{-1}_s (\phi) \Vert_{L^2} \leq C \Vert\phi\Vert_{H^s}.
$$
Since $w_K$ is smooth, it admits a pre-image by $\mathfrak{F}^{-1}_s\circ\mathfrak{F}_\sigma$.
As a consequence, one has:
\begin{eqnarray*}
 |<\phi(t)-\phi(t_0),w_k>_{H^s}|&=&|<\mathfrak{F}^{-1}_s\left(\phi(t)-\phi(t_0)\right),\mathfrak{F}^{-1}_s(w_K)>_{L^2}|\\
&=& |<\mathfrak{F}^{-1}_\sigma\circ\mathfrak{F}^{-1}_s\left(\phi(t)-\phi(t_0)\right),\mathfrak{F}_\sigma\circ \mathfrak{F}^{-1}_s(w_K)>_{L^2}|.
\end{eqnarray*}
Cauchy-Schwarz inequality implies then:
$$
| <\phi(t)-\phi(t_0),w_k>_{H^s}|\leq \Vert\phi(t)-\phi(t_0)\Vert_{H^\sigma}\Vert\mathfrak{F}^{-1}_\sigma(\omega_K)\Vert_{H^\sigma}.
$$
Since $\phi$ is in $C^0(I,H^\sigma (M))$, there exists an open neighborhood $U$ of $t_0$ in $I$ such that, for all $t$ in $U\subset I$:
$$
 |<\phi(t)-\phi(t_0),w_k>_{H^s}|\leq \frac{\epsilon}{2}.
$$
As a consequence, for all $t$ in $U$,
$$
|f(t)-f(t_0)|\leq \epsilon
$$
that is to say that $f$ is continuous in $t_0$.
\end{proof}

One now considers a model case fitting the context of the wave equation. Let $(I \times M, N^2 \ud t^2-g_t)$ be a Lorentzian manifold, $I$ being a compact interval. One defines the energy of the function as being the function of $t$ and $\phi$ in $H^1(M)$ and $\psi$ in $L^2(M)$ by:
$$
E(t,\phi, \psi) = \Vert\phi\Vert^2_{H^1(\{t\}\times M)} + \Vert\psi\Vert^2_{L^2(\{t\}\times M)}.
$$

The final lemma required to achieve the wanted regularity for the the solution of the characteristic Cauchy problem for the wave equation is the following:
\begin{lemma} One assumes that, for all $(\phi, \psi)$ in $H^1(M)\times L^2(M)$, the function:
$$
t\longmapsto E(t, \phi,\psi)
$$
is continuously differentiable in $I$. Let $\phi$ be a function such that:
\begin{itemize}
\item $\phi$ is in $C^0(I,H^1(M)-w)$; as a consequence, $\phi$ lies in $L^\infty(I,H^1(M))$
\item $\partial_t \phi$ is in $C^0(I,L^2(M)-w)$; as a consequence, $\phi$ lies in $L^\infty(I,L^2(M))$.
\end{itemize}
The function:
$$
t\longmapsto E(t, \phi(t),\partial_t\phi(t))
$$
is furthermore assumed to be continuous in $t$.\\
Then $\phi$ and $\partial_t \phi$ are in fact in $C^0(I,H^1(M))$ and $C^0(I,L^2(M))$.
\end{lemma}
\begin{proof} The proof of this fact can be found in the ) of \cite[Chapter 3, Section 8.4, p. 279]{Lions:1972ww}. For the sake of self-consistency, the proof is quoted here, with some adaptations to our framework.

Let $t$ be in $I$ and consider $(t_n)_n$ a sequence of elements of $I$ converging towards $t$ and define the quantity:
\begin{gather*}
\xi_n = \Vert\phi(t)-\phi(t_n)\Vert^2_{H^1(M)} + \Vert\partial_t\phi(t)-\partial_t\phi(t_n)\Vert^2_{L^2(M)}\\
=E(t_n, \phi(t_n), \partial_t \phi(t_n))+E(t, \phi(t), \partial_t \phi(t))\\
 -2 <\phi(t),\phi(t_n)>_{H^1(M)} -2  <\partial_t\phi(t),\partial_t \phi(t_n)>_{L^2(M)}
\end{gather*}
Since $\phi$ is in $C^0(I,H^1(M)-w)$ and $\partial_t \phi$ is in $C^0(I,L^2(M)-w)$, the last two terms converge towards $-2E(t, \phi(t),\partial_t\phi(t))$. Since the energy is assumed to be continuous, the first tow terms converge towards $2E(t, \phi(t),\partial_t\phi(t))$.

The sequence $(\xi_n)$ converges then towards $0$ when $n$ grows. As a consequence, $\phi$ and $\partial_t\phi$ are in $C^0(I,H^1(M))$ and $C^0(I,L^2(M))$, respectively.
\end{proof}

\printbibliography

\end{sloppypar}
\end{document}